\documentclass[paper,10pt]{amsart}
\usepackage{mathdots}
\usepackage{graphicx}
\usepackage[alphabetic]{amsrefs}
\usepackage{breqn}
\usepackage{mathrsfs}

\usepackage{amsmath,amsthm,amssymb,amsfonts,amscd,
epsfig,latexsym,graphicx,epic,eucal,marvosym,textcomp,turnstile,cancel,extarrows} 
\title[Partial Cosine-Funk Transform]{Partial Cosine-Funk Transforms at Poles of the $\Cos^{\lambda}$ Transform on Grassmann Manifolds} 
\author{Adam Cross}
  \address{Department of Mathematics \\
Louisiana State University\\
Baton Rouge, Louisiana}
\email{ccross1@tigers.lsu.edu }
\thanks{A. Cross would like to thank VIGRE@LSU, DMS-0739382, and the US DoED GAANN grant
P200A100080}

\author{ Gestur \'Olafsson}
  \address{Department of Mathematics \\
Louisiana State University\\
Baton Rouge, Louisiana}
\email{olafsson@math.lsu.edu}
\thanks{The research of G. \'Olafsson was supported by NSF grant
DMS-1101337}

\newcommand{\Lie}{\mathrm{Lie}}

\newcommand{\Vol}{\mathrm{Vol}}

\renewcommand{\O}{\mathbf{O}}

\newcommand{\SU}{\text{SU}}

\newcommand{\re}{{\mathfrak{Re}\;}}
\newcommand{\C}{\mathcal{C}}
\newcommand{\K}{\mathbb{K}}

\newcommand{\B}{\mathcal{B}}
\newcommand{\Cos}{\operatorname{Cos}}
\newcommand{\U}{\mathop{\rm{U}}}

\newcommand{\Gr}{\mathrm{Gr}}
\newcommand{\Ad}{\mathop{\rm{Ad}}}

\newcommand{\f}{\mathfrak}

\newcommand{\reals}{\mathbb{R}}
\newcommand{\Exp}{\mathop{\text{\upshape{Exp}}}}

\newcommand{\op}{\cite{MR2854176}}
\newcommand{\ac}{\mathop{\text{a.c.}}}

\usepackage{float}
\restylefloat{table}

\newtheorem{lemma}{Lemma}
\newtheorem{cor}{Corollary}

\newtheorem{thm}{Theorem}
\newtheorem{defn}{Definition}
\newtheorem{prop}{Proposition}

\def\sideremark#1{\ifvmode\leavevmode\fi\vadjust{\vbox to0pt{\vss
 \hbox to 0pt{\hskip\hsize\hskip1em%
 \vbox{\hsize2cm\tiny\raggedright\pretolerance10000
 \noindent #1\hfill}\hss}\vbox to8pt{\vfil}\vss}}}

\begin{document}

\newcommand{\Bk}[1]{\mathcal{B}_{#1}}
\newcommand{\Uk}[1]{U_{#1}}
\newcommand{\Lk}[1]{L_{#1}}
\newcommand{\bk}[1]{\mathfrak{b}_{#1}}
\newcommand{\deltak}[1]{\delta_{#1}}
\newcommand{\betak}[1]{\beta_{#1}}
\newcommand{\fk}[1]{f_{#1}}
\newcommand{\Ck}[1]{\C_{#1}}
\newcommand{\rootsk}[1]{\Sigma^{+}_{#1}}
\newcommand{\kfrack}[1]{\mathfrak{k}_{#1}}
\newcommand{\Kk}[1]{K_{#1}}
\newcommand{\droprank}[2]{#1_{#2}}
\newcommand{\funk}{\mathcal{F}}

\maketitle

\date{}

\begin{abstract}

The cosine-$\lambda$ transform, denoted $\C^\lambda$, is a family of integral transforms we can define on the sphere and on the Grassmann manifolds $\Gr(p, \K^n) = \SU(n,\K)/\text{S}(\U(p,\K) \times \U(n-p,\K))$ where $\mathbb{K}$ is $\reals$, $\mathbb{C}$ or the skew field $\mathbb{H}$ of quaternions.  The family $\C^\lambda$  extends meromorphically in $\lambda$ to the complex plane with poles at (among other values) $\lambda =-1,\ldots, -p$.  In this paper we normalize $\C^\lambda$ and evaluate at those poles.  The result is a series of integral transforms on the Grassmannians that we can view as \emph{partial} cosine-Funk transforms.  The transform that arises at $\lambda = -p$ is the natural analog of the Funk transform  in this setting.\end{abstract}

\tableofcontents

\pagenumbering{arabic}

\section{Introduction}

\noindent
The cosine-$\lambda$ transform is defined for functions on the sphere by $$(\C^\lambda f)(u) = \int_{S^n}|u\cdot v|^\lambda f(v) dv.$$
Integral transforms of this kind have a rich history with connections to many diverse areas of mathematics.  In the case of $\lambda=1$ we have what Lutwak named the ``cosine transform'', noting that $|u \cdot v| =|\cos(\theta)|$ where $\theta$ is the angle between the vectors (\cite{Lutwak}).  For a brief history of the cosine transform and a long list of references, see \cite{ol-rubin}.  Here we offer just a few references to give a sense of it: there are connections to convex geometry (\cite{rubin-zhang}, \cite{Lutwak}, \cite{Gardner}, \cite{Alexandroff}), harmonic analysis and singular integrals (\cite{Ournycheva}, \cite{Ournycheva2006}, \cite{Rubin1998}, \cite{Rubin2002}, \cite{Strichartz}), integral geometry (\cite{Gelfand}, \cite{RubinFractionalCalc}, \cite{RubinFractionalIntegrals}, \cite{Rubin2003}, \cite{Semjanisty}), and others. 

  Of central importance to this paper is the observation that $\C^\lambda$ has a pole at $\lambda = -1$, and that if we normalize and then take the analytic continuation (a.c.)  we get the Funk transform: that is, 
$$\ac_{\lambda=-1} \frac{\Gamma(-\lambda/2)}{\Gamma((1+\lambda)/2)} \C^\lambda f(u) =  c \int_{u^\perp} f(v) dv$$
 where $c$ is computed by setting $f=1$.  In the present paper we will explore similar relationships for a cosine-$\lambda$ transform on the Grassmannian manifolds.  We will see that an appropriate Funk transform on the Grassmannian similarly arises out of the cosine-$\lambda$ transform there, and we will also note some important differences from the case on the sphere.

We consider the cosine transform on the Grassmannian manifolds $\B= \Gr(p,\K^n)$, the manifold of $p$-dimensional subspaces of $\K^n$ where $K = \reals$, $\mathbb{C}$, or the skew field of quaternions $\mathbb{H}$.  We will often use the notation $q= n - p$, and throughout we assume $p \leq q$.
Our methods here are largely based on the techniques and results of the paper \op, in which {\'O}lafsson and Pasquale applied harmonic analysis and representation theory tools to the cosine-$\lambda$ transform. The main result of that paper is to write down the spectrum for $\C^\lambda$ acting on $L^2(\B)$.  We will also use that result in this paper.  

The definition of the transform on Grassmannian manifolds is analogous to the cosine transform on the sphere.  There is a geometrical way to define $|\Cos(\sigma,\omega)|$ on two elements $\sigma,\omega$.  We follow \op \enskip on this.  Write $d$ for the dimension of $\K$ as a real vector space.  We view $\sigma$ as a $dp$-dimensional real vector space and take a convex subset $E\subset \sigma$ containing the zero vector such that the volume of $E$ is 1.  Let $P_\omega:\sigma \rightarrow \omega$ denote orthogonal projection onto $\omega$.  Then we define
$|\Cos(\sigma,\omega)| = \text{Vol}_\reals(P_\omega(E))^{1/d}$.  For more details on this function, in particular to see that it is well-defined, see \op.  From now on we will use it as the appropriate generalization of the $|\cos(\theta)|=|u\cdot v|$ that we used on the sphere.

  Having defined $|\Cos(\sigma,\omega)|$, it makes sense to define the $\C^\lambda$ transform on $L^2(\B)$ by analogy with the sphere:
$$\C^\lambda f(\omega) = \int\limits_{\B}|\Cos( \sigma, \omega )|^{d\lambda}  f(\sigma) d\sigma$$
in the invariant measure.  Our choice to put a $d\lambda$ power on the $|\Cos(\sigma,\omega)|$ (rather than $\lambda$ or some other variant) suits the purposes of this paper.  The reader will find variations on this in other papers.  The choice is largely a matter of convenience to the work at hand.

This $\C^\lambda$ extends analytically to a meromorphic family of transforms.  The first pole of $\C^\lambda$ occurs at $\lambda = -1$ and in this paper we will be interested in the poles $\lambda = -1,\ldots, -p$.  We take an appropriate normalizing function $\gamma(\lambda)$
so that the analytic continuation of $\gamma(\lambda)\C^\lambda$ is entire.  For a function $f\in C^\infty(B)$ and a fixed base point $\beta\in \B$ we compute $$\ac_{\lambda=-1}\gamma(\lambda)\C^\lambda f (\beta)$$ explicitly in coordinates using a familiar integral formula for compact symmetric spaces. 

The striking result of this computation is an integral transform which is itself a certain cosine-$\lambda$ transform on a lower-dimensional Grassmannian manifold $\droprank{\B}{1}$ evaluated at $\lambda =1$.

We consider $\B$ as a symmetric space $K/L$ in the usual way: $K = \SU(n,\K)$ and $L\cong \rm{S}(\U(p, \K)\times \U(q, \K))$, and there is an involution $\tau$ of $K$ such that $L$ is $\tau$-fixed.  In this picture, the base point $\beta$ is $L$, but we will continue to use the notation $\beta$ because we prefer to think of $\beta$ as an element of a Grassmannian, in which case we think of $L$ as the stabilizer of $\beta$.  

We take the eigenspace decomposition of the Lie algebra with respect to $\tau$: that is, $\Lie(K) = \f l \oplus \f q$ where $\f l = \Lie(L)$.  We choose $\f a$ a maximal abelian subspace of $\f q$.  The space $\f a$ has dimension $p$, the rank of $\B$.  It is well known that we may write polar coordinates for $K/L$ using the map $\Phi: L/M \times \f a \to K/L$ defined $(lM,X)\mapsto l\exp(X)L$ where $M$ is the centralizer in $L$ of $\f a$.  We restrict the coordinates to a choice of positive Weyl chamber $\f a^+ \subset \f a$ and then restrict further to a fundamental domain $D^+\subset \f a^+$, which can be parameterized by coordinates $(t_1,\ldots, t_p)$ where $0 \leq t_p \leq \cdots t_1\leq \pi/2$.  The integral $\C^\lambda f (\beta)$ can then be written as an integral over $L \times D^+$.  The final analysis will not depend on our choice of $\f a^+$ because $L$ acts transitively on the Weyl chambers.  

The lower-rank Grassmannian $\droprank{\B}{1}\subset \B$ arises as follows.  Let $\sigma(t_1,\ldots,t_p)$ denote $\Phi(e,(t_1,\ldots,t_p))$, where $e$ is the identity. The fixed base element $\beta\in \B$ is $\sigma(0,\ldots, 0)$.  The function $$|\Cos(\sigma(t_1,\ldots,t_p),\beta)|^{\lambda}\Big |_{\lambda=-1}$$ becomes infinite at $t_1 = \pi/2$ where $|\Cos(\sigma,\beta)|=0$.  Since $L$ is unitary and fixes $\beta$, for any $l\in L$ we have $|\Cos(l\sigma,\beta)|=|\Cos(\sigma,l^{-1}\beta)|=|\Cos(\sigma,\beta)|=0 $.  This leads us to consider the set parameterized by $L \times (\pi/2,t_2,\ldots, t_p)$.  In coordinate-free terms, this set is  $$\{ \sigma \in \B \; | \; |\Cos(\sigma,\beta)|=0\}$$
which we will denote by $Z(\beta)$. 
Clearly $Z(\beta)$ is of interest being the place where $|\Cos(\sigma(t_1,\ldots,t_p),\beta)|^{\lambda}$ blows up at the poles of $\C^\lambda$.

 The set $Z(\beta)$ is not quite the embedded Grassmannian we mentioned.  However, when we choose an appropriate subgroup $\Lk{1}\subset L$, the coordinates  $\Lk{1} \times (\pi/2,t_2,\ldots, t_p)$ parameterize an embedded submanifold diffeomorphic to $\Gr(p-1,\K^{n-2})$.  We call that manifold $\droprank{\B}{1}$.    Note that $\droprank{\B}{1}$ lies in $Z(\beta)$ and we will see that $$Z(\beta) \subset \bigcup_{l\in L}l\droprank{\B}{1}.$$

This $\droprank{\B}{1}$ has its own intrinsic cosine-$\lambda$ transform, which we denote $\Ck{1}^\lambda$.   The main result of this analysis is that
\begin{equation}\label{092406723409}
\ac_{\lambda=-1}\gamma(\lambda)\C^\lambda f (\beta) = c\; \Ck{1}^\lambda f^L (\betak{1}) \big|_{\lambda=1}.
\end{equation}
Here $c$ is a constant computed by putting 1 in for $f$.  Throughout, $f^L(x)=\int_L f(lx)dl$, the integral in unit Haar measure.  The element $\betak{1}$ is a base point in $\droprank{\B}{1}$ analogous to $\beta$.

B. Rubin defined a higher-rank Funk transform for Stiefel manifolds in his paper \cite{rubin-2012}.  His work applies to Grassmannian manifolds by assuming the function lifts to the Grassmannian.  In \cite{ol-rubin} the authors worked out a more specific relationship between Rubin's Funk transform and the cosine-$\lambda$ transform.  Rubin's results are restricted to the case of the field $\reals$, but they are relevant to what we do here, so we explain how our results here fit together with his.

We restate his definition of the higher rank Funk transform in terms of Lie groups.  For a function $f\in C^\infty( \B)$ he defines
\begin{equation}\label{rubinsFunkTransform}
\mathscr{F}f(\beta) = \int_{L}f(l \sigma ) dl
\end{equation}
where $\sigma \subset \beta^\perp$ is arbitrary.  He establishes that \begin{equation}\label{rubinsac}\ac_{\lambda=-p}  \gamma(\lambda)\C^\lambda f(\beta) \propto \mathscr{F} f(\beta).\end{equation}
We agree that his notion of a Funk transform on $\B$ is the appropriate one.  One may think of the classical Funk transform on the sphere as an integral $$\mathscr{F}f(u)=\int_{G=\mathrm{Stab}(u)} f(gv)dg$$
where $v$ is an arbitrary vector in $u^\perp$ and $G$ is a subgroup of the special orthogonal group.  The similarity to (\ref{rubinsFunkTransform}) is clear since $L=\mathrm{Stab}(\beta)$.

Observe that on the sphere, we might write $|\Cos(u,v)|$ for $|u\cdot v|$ since this is the natural cosine between two elements.  However, on the sphere, the conditions $|\Cos(u,v)|=0$ and $u\subset v^\perp$ are equivalent, but in a Grassmannian the condition $\sigma \in \beta^\perp$ is a stronger condition than $|\Cos(\sigma,\beta)|=0$.  The latter amounts to the statement that $\sigma$ contains a vector orthogonal to $\beta$ (and vice versa).  That distinction is of paramount importance to the present paper.

Using Rubin's result, we characterize $\C^\lambda f (\beta)$ at the poles $-1, -2, \ldots, -p$ for the real case, which ties together our results in this paper with his result at $\lambda=-p$.  Having established that $\gamma(\lambda)\C^\lambda f(\beta)\Big |_{\lambda=-1}$ yields a cosine-$\lambda$ transform of $f$ over an embedded Grassmannian of rank $p-1$, we further establish that  $\gamma(\lambda)\C^\lambda f(\beta)\Big |_{\lambda=-2}$ yields a cosine-$\lambda$ transform on a rank $p-2$ embedded Grassmannian,  and so on stepping down in rank at each pole until at $\lambda=-p$ we have Rubin's Funk transform.  

This situation will be clearer to the reader once we have written $\C^\lambda f(\beta)$ in coordinates,  but to give a general idea, the stepping down will look something like this.  We start with an integral over $L\times D^+$.  Then, at the first pole, we have an integral over $L \times D^+ \Big |_{t_1=\pi/2}$:
one vector in $\sigma(\pi/2, t_2, \ldots, t_p)$ is perpendicular to $\beta$.  Then at $\lambda = -2$, we have $L\times D^+ \Big |_{t_1=t_2=\pi/2}$: two independent vectors in $\sigma$ are perpendicular to $\beta$, and we continue until we reach $L\times (\pi/2,\ldots, \pi/2)  $.  In the last expression, $\sigma(\pi/2,\ldots, \pi/2)\in \B$ is contained in $\beta^\perp$.

At each pole from -1 to -p we make a step down toward the Funk transform.  By contrast, the sphere does not admit any division of its Funk transform into steps like this in quite so natural a way.  Perhaps we may think of these intermediate cosine-$\lambda$ transforms on embedded Grassmannians as some kind of \textit{partial} cosine-Funk transforms.  We leave that to the reader to consider.

Let us at the end mention the connection of those results to representation theory. 
Let $G=\mathrm{SL}(n,\mathbb{K})$. Then $G$ acts on $\mathcal{B}$ in a natural way
$g\cdot \beta =\{g(v)\mid v\in\beta\}$ and $\mathcal{B}=G/P$ where $P=MAN$ is a maximal
parabolic subgroup in $G$. It was shown in \cite{MR2854176} that $\mathcal{C}^{\lambda-\rho}$, where
$\rho =d(n+1)/2$ is an $G$-intertwining
operator between two representations $\pi_{\lambda}$ and $\pi_{-\lambda}\circ \theta$. For a representation
$\mu$ of $K$ let $\eta_\mu (\lambda -\rho)$ be the eigenvalue of $\mathcal{C}^{\lambda-\rho}$ on
$L^2_\mu (\mathcal B)$, the space of $L^2$-functions of type $\mu$. Then the zeros and poles of the
sequence $\{\eta_\mu (\lambda -\rho)\}$ given information about the composition series for $\pi_\lambda$ and
$\pi_{-\lambda}\circ \rho$. This idea was used in \cite{MS14} to determine the composition series explicitly. Our 
results then give extra information about intertwining operator onto the quotient and a geometric interpretation
of the image, respectively the kernel.

\subsection{Notation}

  The constant $p$ is fixed throughout as the rank of the base Grassmannian $\B$.  Since we consider embedded submanifolds that are Grassmannians of lower rank, in several instances we will use a subscript $k$ to denotes that we are considering an element in the rank $p-k$ setting.  For the coordinates, we write $\textbf{t}_k=(t_{k+1},\ldots, t_p)$.

\subsection{Outline}

In Section \ref{sec1} we recall some basic results from harmonic analysis including an integral formula for compact symmetric spaces in polar coordinates.  We also summarize some of the results from the paper \op \; and establish a few elementary corrolaries.  

In Section \ref{sec2} we compute the transform that arises from $\C^\lambda$ at its first pole $\lambda=1$.  We use the integral formula from Section \ref{sec1} to write $\C^\lambda$ in coordinates.  This yields an integral transform we have called $\funk_1$, and we call this a ``partial cosine-Funk transform'' on the Grassmann manifold.  

Since the result of taking this limit is an integral in coordinates, we spend some time in Subsection \ref{geometry} examining the geometric interpretation of $\funk_1$.

Next we observe that $\funk_1$ is an intertwining operator for the left regular representation on $\mathcal{C}^\infty(\B)$ and we compute the image and kernel of $\funk_1$.

In Section \ref{higherPolessec} we consider the poles $\lambda=-1,\ldots, -p$ of $C^\lambda$.  Unlike the work in previous sections, in this section we rely on a result proved by B. Rubin.  The argument presented here may appear to subsume our work on the first pole, but in fact it is quite different because in our analysis of the first pole we did not use Rubin's result, and his methods are quite different from ours.

\section{Background}
\label{sec1}
\noindent
In this section we will first recall some of the basic notation and results we will use in the present paper.  We use a well known integral formula for compact symmetric spaces which can be found in S. Helgason's \textit{ Groups and Geometric Analysis}.  After that, we specialize to the cosine-$\lambda$ transform on Grassmannian manifold where we establish the notation we will use and review some of the results from \cite{MR2854176}.  Here we have also included some small corollaries concerning poles of the cosine-$\lambda$ transform that follow quickly from \cite{MR2854176}.

\subsection{Symmetric Space Integral Formula}

The setting is a symmetric space $K/L$ of compact type.  Here $K$ is a compact Lie group and we have an involution $\tau$ of $K$ such that $L$ is the $\tau$-fixed subgroup of $K$.  We recall an integral formula for this space based on a kind of polar coordinate decomposition.  The details can be found in \cite{MR1790156}, starting on page 187.  

Let $\f k$ be the Lie algebra of $K$, $\f l$ the Lie algebra of $L$.  Then $\f l$ is the +1 eigenspace of the derived involution $\tau:\f k \to \f k$.  Let $\f q$ be the -1 eigenspace.  Then $\f k = \f l \oplus \f q$.  Following Helgason's treatment, let $\f a$ be a maximal abelian subspace of $\f q$, and let $M$ denote the centralizer of $\f a$ in $L$.  The group $A = \exp \f a$ is a closed subgroup of $K$.  

Note, then, that $\f g_0  = \f l \oplus i \f q$ is a non-compact real form of the complexification $\f k^{\mathbb{C}}$ and $\f l \oplus i \f q$ is a Cartan decomposition.  The space $i\f a$ is then a maximal abelian subspace of $i \f q$ and we let $\Sigma$ denote the set of restricted roots of $\f g_0$ with respect to $i\f a$.  Given some choice of positive Weyl chamber, we let $\Sigma^+$ denote the set of positive roots.

The polar coordinate map $\Phi$ is defined
\begin{align*}
\Phi: L/M \times A &\rightarrow K/L\\
(kM,a) & \mapsto kaL.
\end{align*}
This map is onto and $|\det(d\Phi_{(kM,b)})|=\prod\limits_{\alpha\in\Sigma^+} |\sin \alpha(i\log(b))|^{m_\alpha}$, where $m_\alpha$ denotes the multiplicity of $\alpha$.    Let $\delta(b):= \\ \prod\limits_{\alpha\in\Sigma^+} |\sin \alpha(i\log(b))|^{m_\alpha}$.  Let $A'$ be the set $\{x\in A | \delta(x) \neq 0\}$.  Also, let $(K/L)_r$ be the complement of the singular set in $K/L$.  Then $\Phi:(L/M)\times A'$ maps onto $(K/L)_r$ and this map is regular.  We can therefore write
$$\int_{K/L} f(kL) dk =c\int_{L} \int_{A'} f(la L) \delta(a) da dl$$
for a constant $c$.  

\begin{prop}\label{polarCoordsProp237459}
Fix a particular Weyl chamber $\f a^+$ and let $A^+ = \exp \f a^+$.  Then 
$$\int_{K/L} f(kL) dk = c \int_{L} \int_{B^+} f(lb L) \delta(b) db dl$$
for a constant $c$.  
\end{prop}

In each case, the constant $c$ is determined by letting $f=1$ so that in Prop. \ref{polarCoordsProp237459}, $c = 1/\int_{B^+} \delta(b)db$.

\begin{proof}
Observe that the Weyl group can be represented in $L$ and it permutes the Weyl chambers.
\end{proof}

Throughout, let us write $f^L$ for the function given by 
$$f^L(x) = \int_L f(lx) dl$$
with the integral taken with respect to unit Haar measure.

\subsection{Cosine-$\lambda$ Transform on $\Gr(p,\mathbb{K}^n)$}
Here we establish some background and notation, and mention some of the essential results from \op \enskip we will be using.  

We specialize to the setting $\B = \Gr(p,\mathbb{K}^n)$, the Grassmannian manifold of $p$-dimensional subspaces in $\mathbb{K}^{p+q}$ where $\mathbb{K}$ is $\reals$, $\mathbb{C}$ or the skew field $\mathbb{H}$ of quaternions.  We will assume $q\geq p \geq 2$.  Let $n=p+q$.  Let $\{e_1,\ldots ,e_n\}$ be an ordered orthonormal basis for the underlying space $\K^n$.  We set $K=SU(p+q,\K)$ and $L=S(U(p,\K)\times U(q,\K))$.  Then $\B \cong K/L$, which is a compact symmetric space with involution $$\tau(x) =
\left(\begin{smallmatrix}
I_p & 0\\
0 & -I_q
\end{smallmatrix}\right)
x
\left(
\begin{smallmatrix}
I_p & 0\\
0 & -I_q
\end{smallmatrix}\right)$$
and $K^\tau = L$.

\subsubsection{Lie Algebra Decomposition and Simple Roots}

We take the decomposition $\f k = \f l + \f q$ with $\f l$ and $\f q$ the +1 and -1 eigenspaces of $\f k$ with respect to the infinitesimal involution $\tau$.  Then 

$$\f q = \left\{ Q(X)=\left(\begin{array}{cc} 0_{pp} & X \\ -X^\ast & 0_{qq}\end{array}\right) | X\in M(p\times q,\K) \right\}.$$

We also write down a maximal abelian subspace of $\f q$ and coordinates for it.  Here our choice will differ from the one in \op \enskip by a conjugation.    Let $E^{(r,s)}_{\nu,\mu} = (\delta_{i\nu}\delta_{j\mu})_{i,j}$, the matrix in $M(r\times s,\K)$ with all entries equal to 0 but the $(\nu,\mu)$th, which equals 1.  Let $\textbf{t}=(t_1,\ldots, t_p)$ and 
$$X(\textbf{t}) = -\sum_{j=1}^p t_j E_{j,p+q+1-j}^{(p,q)}$$
and
$$Y(\textbf{t}) = Q(X(\textbf{t}))\in \f q.$$
  Then $\f b = \{ Y(\textbf{t}) | \textbf{t}\in \reals^p\}$ is a maximal abelian subspace of $\f q$ and  

\begin{equation}\exp Y(\textbf{t}) = \left(
\begin{array}{ccc}
\begin{smallmatrix}
 \cos(t_1) &&\\
& \ddots &\\
&&\cos(t_p)\end{smallmatrix}
 & 0 &  \begin{smallmatrix}
 && -\sin(t_1) \\
& \iddots &\\
-\sin(t_p) &&\end{smallmatrix} \\
0 & I_{q-p} & 0 \\
\begin{smallmatrix}
 && \sin(t_p) \\
& \iddots &\\
\sin(t_1) &&\end{smallmatrix}& 0 &  \begin{smallmatrix}
 \cos(t_p) &&\\
& \ddots &\\
&&\cos(t_1)\end{smallmatrix} 
\end{array}
\right).\label{298777420984379} \end{equation}

Denote by $\Sigma_{\f k}$ the set of roots of $\f b_{\mathbb{C}}$ in $\f k_{\mathbb{C}}$ and let $\Sigma_{\f k}^+$ denote the positive roots with respect to some choice of ordering.  Below we will make this choice explicit.  

Let us define $\{ \epsilon_j \}$ as the basis of $\mathfrak b^*$ dual to $\{ Y^j =Y(\sum_{m=1}^p \delta_{j,m}t_m) \}$ so that $\epsilon_j(Y(\textbf{t} )) = it_j$.

  The following proposition is Lemma 5.2 of \op.
\begin{prop} \label{rootsProp}
  The roots are
$$\Sigma_{\f k} = \{ \pm \epsilon_i \pm \epsilon_j \;(1 \leq i \neq j \leq p, \pm\text{ independent), }\pm \epsilon_i \; (1 \leq i \leq p), \pm2\epsilon_i \; (1 \leq i \leq p)\}$$ 
with multiplicities, respectively, $d$ (and not there in case $p=1$), $d(q-p)$ (and not there in case $p=q$) and $d-1$ (and not there if $d=1$). 
\end{prop}

Let us pick a simple system of roots to work with in each case.  We indicate this choice and the corresponding positive Weyl chamber for each case in Table \ref{table1}.

\begin{table}[H]
\caption{Systems of Simple Roots}
\centering
\begin{tabular}{ l p{4cm} l }
Case& Simple System of Roots & Positive Weyl chamber\\
$p=q$,$d=1$ &
 $\{ \epsilon_1-\epsilon_2, \epsilon_2-\epsilon_3,\ldots,   \epsilon_{p-1}-\epsilon_p, \epsilon_{p-1}+\epsilon_p \}$ 
 & $|t_p| < t_{p-1} < \cdots < t_1$\\ \hline
 $p\neq q$,$d=1$ & 
 $\{ \epsilon_1-\epsilon_2,  \epsilon_2-\epsilon_3,\ldots,   \epsilon_{p-1}-\epsilon_p,\epsilon_p \}$ &
 $0<t_p < t_{p-1} < \cdots < t_1$ \\ \hline
 $p \neq q$, $d>1$ &
 $\{ \epsilon_1-\epsilon_2,  \epsilon_2-\epsilon_3,\ldots,   \epsilon_{p-1}-\epsilon_p,\epsilon_p \}$ & 
 $0<t_p < t_{p-1} < \cdots < t_1$ \\ \hline
 $p = q$, $d>1$&
 $\{ \epsilon_1-\epsilon_2,  \epsilon_2-\epsilon_3,\ldots,   \epsilon_{p-1}-\epsilon_p, 2\epsilon_p \}$&
 $0<t_p < t_{p-1} < \cdots < t_1$
 \end{tabular}
 \label{table1}
\end{table}

\subsubsection{Highest Weights and Spherical Representations}

  Let $\widehat{K}$ denote the set of equivalence classes of irreducible representations of $K$, and let $\widehat{K}_L$ denote the subset of $\widehat{K}$ of $L$-spherical representations.  Define the notation
$$\Lambda^+ := \left\{  \mu \in i\f b^\ast \big \vert   ( \forall \alpha \in \Delta^+_{\f k})  \frac{\langle \mu,\alpha\rangle}{\langle \alpha, \alpha \rangle } \in \mathbb{N}_0  \right\}.$$
The map $\pi \mapsto$ (highest weight of $\pi$) maps $\widehat{K}_L$ injectively into $\Lambda^+$.  This map is bijective if $\B$ is simply connected (\cite{MR1790156}, p. 535).  In general, $\widehat{K}_L$ maps to a subset $\Lambda^+(\B)$ of $\Lambda^+$ (\cite{MR2831149}).  

\begin{prop}
If $\K = \reals$ and $p=q$, then 
$$\Lambda^+ = \left\{ \mu = \sum_{j=1}^p m_j \epsilon_j \Big\vert  m_j \in \mathbb{Z}, m_i-m_{i+1} \in 2\mathbb{N}_0 \text{ for } i=1,\ldots,p, m_{p-1}>|m_p| \right\}.$$
Observe that as a consequence, $m_1 \geq m_2 \geq \cdots m_{p-1} \geq |m_p|$.
If $\K = \reals$ and $p\neq q$, then 
$$\Lambda^+ = \left\{ \mu = \sum_{j=1}^p m_j \epsilon_j \Big\vert  m_j \in  \mathbb{N}_0, m_i-m_{i+1} \in 2\mathbb{N}_0 \text{ for } i=1,\ldots,p \right\}.$$
In the other cases,
$$\Lambda^+  = \left\{ \mu = \sum_{j=1}^p m_j \epsilon_j \Big\vert  m_j \in 2\mathbb{N}_0 \text{ and } m_1 \geq m_2 \geq \cdots m_{p-1} \geq m_p \right\}.$$
\end{prop}

\begin{prop}
In the cases $\K=\mathbb{C}$ and $\K = \mathbb{H}$, the subset $\Lambda^+(\B)$ of $\Lambda^+$ given by highest weights of irreducible spherical representations is the full set $\Lambda^+$.  

For the case $\K=\reals$, an element $\mu = \sum m_i\epsilon_i \in\Lambda^+$ is in $\Lambda^+(\B)$ if and only if  $m_i \in 2 \mathbb{Z}$ for $i=1,\ldots,p$.  
\end{prop}

\subsubsection{$\C^{\lambda}$ as an Intertwining Operator and its Spectrum}

The following definition follows \op \enskip except that we use a different exponent on the $|\Cos(x,\omega)|$ factor.  This is a matter of convenience for the work at hand.  See \cite{ol-rubin} for further remarks.
    
\begin{defn}
Let $d$ be the dimension of $\mathbb{K}$ over $\reals$.  On the space $\B = \Gr(p,\mathbb{K}^n)$ the Cosine-$\lambda$ transform is defined for $\re (d\lambda) >-1$ and $f\in L^2(\B)$ by
$$\C^\lambda f(\omega) = \int_{\B} | \Cos(x, \omega)|^{d\lambda} f(x) dx.$$
\end{defn}
See the introduction for a definition of $|\Cos(x,\omega)|$ and refer to \op \enskip for a detailed explanation.
\begin{thm}
The $\C^\lambda$ transform extends to a meromorphic family of intertwining operators $\C^\lambda :\C^\infty(\B) \rightarrow \C^\infty(\B)$.
\end{thm}
This is Theorem 4.5 (1) of \op.

The space $L^2(\B)$ decomposes into  
\begin{equation}L^2(\B) \cong_K \bigoplus_{\mu \in \Lambda^+(\B)} L^2_\mu(\B)\label{5772398732477098098}
\end{equation}
where $L^2_\mu(\B)$ is an irreducible subrepresentation of $K$ with highest weight $\mu$, and $\C^\lambda$ acts by scalar on each of these spaces  $L^2_\mu(\B)$.  We therefore speak of the $K$-spectrum of $\C^\lambda$, meaning the set $\{ \eta_\mu(\lambda) | \mu\in \Lambda^+(\B)\}$ where $\eta_\mu(\lambda)$ is the scalar such that $\C^\lambda\restriction_{L^2_\mu(\B)} = \eta_\mu(\lambda)\text{id}_{L^2_\mu(\B)}$.  We will identify $\mu$ with the $p$-tuple $(m_1,\ldots, m_p)$ where $\mu = \sum_1^p m_i \epsilon_i$.

The next theorem is Theorem 5.11 of \op \enskip with notation changes to suit this paper.  It is one of the primary results of that paper.  We will use it in this paper to investigate the poles of the $\C^\lambda$ transform.

\begin{thm}
\label{spectrumTheorem}
The $K$-spectrum of the Cosine-$\lambda$ transform is

$$\eta_\mu(\lambda) = (-1)^{|\mu|/2} \frac{\Gamma_{p,d}(\frac{1}{2}dn ) }{\Gamma_{p,d}(\frac{1}{2}dp)} \frac{\Gamma_{p,d} (\frac{1}{2}(d\lambda+dp)) \Gamma_{p,d}(\frac{1}{2}(-d\lambda+\mu))}{\Gamma_{p,d}(-\frac{1}{2}d\lambda) \Gamma_{p,d}(\frac{1}{2}(d\lambda+dn+\mu))}.$$
\end{thm}
The function $\Gamma_{p,d}(\lambda)$ is defined for a $p$-tuple $\lambda \in \mathbb{C}^p$ by
$$\Gamma_{p,d}(\lambda) = \prod_{j=1}^p \Gamma(\lambda_j - \frac{d}{2} (j-1)).$$  When $\lambda \in \mathbb{C}$, the notation $\Gamma_{p,d}(\lambda)$ means $\Gamma_{p,d}(\lambda,\ldots,\lambda)$.

\begin{cor}
If $d=1$ or 2, the cosine-$\lambda$ transform has poles at the negative integers -1, -2, -3, \ldots except in the case $p=1$, $d=1$ where the poles occur only at the odd negative integers.  If $d=4$, then the poles are $\lambda =-1, -3/2, -2, -5/2, -3, \ldots$
\end{cor}

\begin{cor}
In the case $d = 1$, the poles $\lambda =-i$ for $i=1, \ldots, -p$ have orders $i/2$ for $i$ even, $(i+1)/2$ for $i$ odd.  
\end{cor}

\begin{cor}
In the cases $d=1$, $d=2$ and $d=4$, the Cosine-$\lambda$ transform has a simple pole at $\lambda = -1$.  

\end{cor}

Throughout, we can use the function $$\gamma(\lambda) = \frac{1}{\Gamma_{p,d} (\frac{d}{2}(\lambda+p))}$$
to normalize $C^\lambda$ so that $\gamma(\lambda) C^\lambda$ is holomorphic.

\section{The First Pole of $C^\lambda$}

\label{sec2}

\noindent
In this section we write down the $C^\lambda$ transform in coordinates using well known harmonic analysis tools.  Then we compute the limit of the normalized trasnform 
$$\lim_{\lambda\to -1} \gamma(\lambda) C^\lambda (f) $$
where $ C^\lambda $ has its first pole in our notation.  This limit yields an integral transform $\droprank{\funk}{1}$ which we also write explicitly in coordinates.  We explore the geometric interpretation of this transform in detail and show that we can view $\droprank{\funk}{1}$ as a cosine-$\lambda$ transform on an embedded submanifold that is diffeomorphic to a rank-$(k-1)$ Grassmannian.  To minimize clutter in the notation, we have used a subscript 1 rather than $k-1$. 

This transform is also an intertwining operator for the left regular representation of $K$ on $\C^\infty(\B)$, and we explicitly compute its image and kernel.

We have been thinking of this transform $\droprank{\funk}{1}$ as a kind of {\emph partial } cosine-Funk transform which makes sense on the higher rank Grassmanians but does not exist on the sphere.

\subsection{Weyl Chambers and Fundamental Domains}

We let $\beta$ denote $\{(x_1,\ldots, x_p,0,\ldots, 0)|x_1, \ldots x_p\in \K\}$ and
let $\{b_1,\ldots, b_p\}$ be the standard basis.  Then for any $\omega = k \beta$ where $k\in K$, because $k$ is an element of the orthogonal we have $$|\Cos(\sigma, \omega)|=|\Cos(k^{-1} \sigma , \beta)|.$$  Then $$\int_{\B} | \Cos(\sigma, \omega)|^{d\lambda} f(\sigma) d\sigma = \int_{\B} | \Cos(\sigma, \beta)|^{d\lambda} (L_{k^{-1}}f)(\sigma) d\sigma$$
so it suffices to consider the cosine transform evaluated at $\beta$.

Let $\alpha(b) = |\Cos(b\cdot \beta, \beta)|$ (see Theorem 4.2 from \op) and note that $\alpha(b)$ is $L$-invariant.

The following is Lemma 5.8 from \op:
\begin{lemma}
Let $\textbf{t}\in\reals^p$ and $\lambda \in\mathbb{C}$.  Then $\alpha(\exp Y(\textbf{t}))^\lambda = \prod_{j=1}^p |\cos(t_j)|^\lambda$.  
\end{lemma}

Because of this, we write $\alpha(\textbf{t})=\prod_{j=1}^p |\cos(t_j)|^\lambda$ and  $\alpha_k(\textbf{t}_k)=\prod_{j=k+1}^p |\cos(t_j)|^\lambda$

\begin{prop}  As before, take $B^+ = \exp \f b^+$ where $\f b^+$ is a positive Weyl chamber.  Then
\begin{align}
\C^\lambda f(\beta) & =  c \int_{L} \int_{B^+}  \alpha( b)^{d\lambda} f(l b \beta) \delta(b) db \; dl\\
 & =  c \int_{B^+}  \alpha( b)^{d\lambda} f^{L}(b \beta) \delta(b) db 
\end{align}
where $c= \C^\lambda 1(\beta)/ \int_{B^+}  \alpha( b)\delta(b) db$.  This $c$ can be calculated.
\end{prop}
\begin{proof}
Apply Proposition \ref{polarCoordsProp237459}.
\end{proof}

We will now play somewhat loose with the constant $c$, which may vary from line to line.  In each case, it can be evaluated explicity by setting $f=1$.

\begin{prop}
 Fix a fundamental domain 	$U\subset \f b^+$ for the map  $\exp: \f b^+ \rightarrow B^+$.  Then for an $L$-invariant $f$, 
\begin{equation}\C^\lambda f(\beta) = c \int_{U} \alpha(\exp X)^{d\lambda} f(\exp X \beta)\delta(\exp X) dX.\end{equation}
Using the map $Y$ as a coordinate chart, we have
\begin{equation}\C^\lambda f(\beta) = c \int_{Y^{-1}(U)} \alpha(\exp Y(\textbf{t}))^{d\lambda} f(\exp Y(\textbf{\textbf{t}}) \beta)\delta(\exp Y(\textbf{t})) d\textbf{t}.\end{equation}
\end{prop}

\begin{prop}\label{givenAFundamentalDomain}Given a fundamental domain $U$ for the map $\Exp \circ Y : \reals^p \mapsto B^+$ and the same $f$ as above, we have
\begin{equation}\label{239872940072308}\C^\lambda f(\beta) = c \int\limits_{U}  \prod_{j=1}^p |\cos(t_j)|^{d\lambda} f(\exp Y(\textbf{\textbf{t}}) \beta)\delta(\exp Y(\textbf{t})) d\textbf{t}.\end{equation}

\end{prop}

\begin{lemma}  The map  $\Psi: (t_1,t_2,\ldots,t_p)\mapsto \exp Y(t_1,\ldots,t_p) b_0$ is $\pi$-periodic in each $t_i$ and if $\Psi(t) = \Psi(s)$ then for each $i$, $s_i = t_i + k_i \pi$ where $k_i$ is an integer depending on $i$.
\end{lemma}

 \begin{proof}
Suppose $\exp Y(\textbf{t}) \beta = \beta$ and $e_i(\textbf{t})$ is the $i$th column of $\exp Y(\textbf{t})$.  Then $e_i(\textbf{t})$ is linearly independent of $\{ b_j | j\neq i\}$, so $e_i(\textbf{t}) = b_i$, which forces $t_i=k_i \pi$ for some integer $k_i$.  

\end{proof}

Now let us deal with the action of the Weyl group.  Consider a $p$-cube centered at the origin: $$D =[-\pi/2,\pi/2]\times \cdots \times [-\pi/2,\pi/2].$$
Let $D^+$ denote the intersection of $D$ with the positive Weyl chamber, and let $D'$ denote the subset of $D$ containing just the regular elements.

\begin{lemma} Up to a set of measure zero
$$D' = \bigcup_{g\in W} gD^+,$$
and W acts simply transitively on the connected components of $D'$.
\end{lemma}

\begin{proof}
The Weyl group acts simply transitively on the Weyl chambers and $D'$ is $W$-invariant.  
\end{proof}

\begin{prop}
\label{fundDomainProp}
The set $D^+ = \f b^+\cap D$ is a 
 fundamental domain as in Proposition \ref{givenAFundamentalDomain}.
\end{prop}

\subsection{Cosine-$\lambda$ Transform in Coordinates}

We will now write the $\C^\lambda$ transform concretely in coordinates.  The basic form we use comes from (\ref{239872940072308}) in Proposition \ref{givenAFundamentalDomain}.  Two parts of (\ref{239872940072308}) need to be specified: $\delta$ and $U$.  Both depend on the particular case---that is, they depend on the value of $d$ and whether $p=q$.

Recall that for brevity we defined $\textbf{t}_k = (t_{k+1},\ldots, t_p)$.  Then we will write $f(\textbf{t}) = f(\exp Y(t_1,\ldots,t_p)\beta)$ and $f_k(\textbf{t}_k)=f(\pi/2,\ldots,\pi/2,t_{k+1},\ldots, t_p)$ (the first $k$ arguments are $\pi/2$).

Recall that $f$ is $\pi$-periodic in all variables and $\delta(b):=\prod\limits_{\alpha\in\Sigma^+} |\sin \alpha(i\log(b))|^{m_\alpha}$.  Applying Proposition \ref{rootsProp} for the roots, we have 
\begin{dmath}\delta(\textbf{t}) = \delta(\exp Y(\textbf{t})) = \prod_{i=1}^p [ |\sin t_i|^{d(q-p)} |\sin (2t_i)|^{d-1} ]\prod_{1\leq i < j \leq p} |\sin(t_i - t_j)\sin(t_i+t_j)|^d.\end{dmath}

Since $\sin(2u) = 2\cos u \sin u$ and $|\sin(u-v)\sin(u+v)| = |\cos^2 u - \cos^2 v|$, this simplifies.  We set the notation
\begin{align*}
 &\deltak{k}(\textbf{t}_{k})\\
= &\prod_{i=k+1}^p [ |2 \cos t_i |^{d-1}    |\sin t_i|^{d-1+d(q-p)} ]\prod_{k+1\leq i < j \leq p} | \cos^2 t_i - \cos^2 t_j |^d
\end{align*}
because we will need to consider the lower rank $\delta$s later.

To write down the region $U$ in each case, we refer to Table \ref{table1} and apply Proposition \ref{fundDomainProp}.  The case $p=q,d=1$ is different from the other three cases because the positive Weyl chamber has a different shape.   

For the case $p=q,d=1$, one can see that $U=\{ (t_1,\ldots,t_p) | |t_p|< t_{p-1} < \cdots < t_1<\pi/2\}$ is a fundamental domain.  The integral over this region can therefore be written

\begin{equation}\C^\lambda f(\beta) =c \int\limits_{0}^{\pi/2}  \int\limits_{0}^{t_1}  
  \cdots  \int\limits_{-t_{p-1}}^{t_{p-1}}  \alpha(\textbf{t})^{d\lambda} f(\textbf{t})\delta(\textbf{t}) dt_p \cdots dt_1.\end{equation}
Observe that $\textbf{t}\mapsto  \alpha(\textbf{t})^{d\lambda} \delta(\textbf{t})$ is even in each variable, so if $f$ is odd in $t_p$ then the integral is zero.  We will therefore assume $f$ is even in $t_p$, so  

  \begin{dmath}\label{4572989823} \label{othercases}
  \C^\lambda f(\beta) = c \int\limits_{0}^{\pi/2}  \int\limits_{0}^{t_1}  
  \cdots  \int\limits_{0}^{t_{p-1}}  \alpha(\textbf{t})^{d\lambda} f(\textbf{t})\delta(\textbf{t}) dt_p \cdots dt_1.
\end{dmath}

In the other three cases, the fundamental domains are all $$U=\{ (t_1,\ldots,t_p) | 0<t_p< t_{p-1} < \cdots < t_1<\pi/2\},$$  and the integral takes the same form  (\ref{4572989823}).

We perform a change of variable $u_i = \cos(t_i)^2$ and define $\textbf{u}=(u_1,\ldots, u_p)$.  Then we have 

\begin{dmath}
\label{CinCoords}
c
\int\limits_{0}^{1}  \int\limits_{u_1}^{1}  
  \cdots  \int\limits_{u_{p-1}}^{1}  \prod_{i=1}^p  u_i^{(d(\lambda+1)-2)/2}  f^\vee(\textbf{u}) \nu_p(\textbf{u}) du_p\cdots du_1
 \end{dmath}
 where $f^{\vee}(\textbf{u}) = f(\cos^{-1}\sqrt{u_1},\ldots, \cos^{-1}\sqrt{u_p})$ and $\nu_{k}^m(u_{k+1},\ldots,u_m) = \prod_{i=k+1}^m  (1-u_i)^{(d-2 +d(q-p))/2}  \prod_{k+1\leq i<j \leq m } |u_i- u_j|^d$ (and $\nu_k = \nu_k^p$).  Note that $\nu_k$ is not just $\delta_k$ after a change of variable.  Rather, we have collected all the factors of $u_i$, so none appear in $\nu_k$.  Also, the quantity $q-p$ which appears will always be the same in our work even as $k$ and $m$ change in $\nu_k^m$.  
 
The integral (\ref{4572989823}) is independent of our choice of Weyl chamber and ordering.  All such choices are conjugate under $L$, and we assume $f$ is bi-$L$-invariant.  More explicitly, had we chosen  some other maximal abelian $\f b' \subset \f q$ and positive Weyl chamber $(\f b')^+$, it is well known that there is an element $l\in L$ giving an automorphism $\Ad(l)$ so that $\f b\mapsto \f b'$ and $\f b^+ \mapsto (\f b')^+$.  All our work throughout this text is independent of this choice.

\subsection{$\droprank{\funk}{1}$ in Coordinates }

We now take the limit of the normalized $\gamma(\lambda)\C^\lambda(f)(\beta)$ as $\lambda$ goes to -1.  That is the location of the first pole of $\C^\lambda(f)(\beta)$, and this computation yields $\droprank{\funk}{1}$.

Let
\begin{dmath}\label{defF} F(\lambda,u_1) = \int\limits_{u_1}^{1}  
  \cdots  \int\limits_{u_{p-1}}^{1}  \prod_{i=2}^p  u_i^{(d(\lambda+1)-2)/2}  f^\vee(\textbf{u}) \nu_{1}(\textbf{u}) \\  \prod_{1 <j \leq p } |u_1- u_j|^d  du_p\cdots du_2.
  \end{dmath}
Then up to a constant factor $c$, the limit $\lim_{\lambda \to -1}  \gamma(\lambda)\C^\lambda f (\beta)$ 
equals

$$\lim_{\lambda\to-1} \gamma(\lambda)
\int\limits_{0}^{1} u_1^{(d(\lambda+1)-2)/2} (1-u_1)^{(d-2 +d(q-p))/2} F(\lambda,u_1) du_1. $$

Recall that the order of the pole at $\lambda=-1$ is 1.  Therefore, for this computation we may replace $\gamma(\lambda)$ with $\Gamma(\lambda)$ for simplicity.
  
The difficulty we must overcome now is the dependence of $F$ on $\lambda$.  Limits of the form 
$$\lim_{\lambda\to -1} \Gamma(\lambda)\int_0^{1} u^\lambda f(u) du$$
are comparatively trivial, but we must carefully deal with the interior $\lambda$s in $F$.  
This difficulty is the purpose of the following lemma.

\begin{lemma}
\label{tech_prop_1}
Let $f(\lambda,s, t)$ be a function that is bounded on $[-1,0]\times[0,1]\times [0,1]$ and assume that $f(\lambda,s, t)$ converges to $f(\lambda,0, t)$ uniformly in $\lambda$ as $s\to 0$.  Define
$$F(\lambda,s) = \int\limits_{s}^1 t^\lambda |s-t|^d f(\lambda,s,t)\mu(t)dt$$
where $d\geq 1$ and $\mu$ is integrable on $[0,1]$ and bounded on some $[0,\eta]$, $\eta>0$.  Then $F(\lambda,s)$ converges to $F(\lambda,0)$   uniformly in $\lambda$ as $s\to 0$ and $F$ is bounded.    

\end{lemma}

\begin{proof}
  The essential observation is that $u_p^d - |u_p-u_{p-1}|^d$ is $O(u_{p-1})$ as $u_{p-1}\to 0$ and $\int_{u_{p-1}} u_p^{-1}du_p \mu(u_p)$ is $O(\ln(u_{p-1}))$.

\end{proof}

\begin{lemma} For $\lambda \in [-1,0]$, as $u\to 0$ the function $F(\lambda,u)$ from (\ref{defF})  converges to $F(\lambda,0)$ uniformly in $\lambda$ and $F$ is bounded.
\end{lemma}

\begin{proof}  
Apply Lemma \ref{tech_prop_1} repeatedly to $f^\vee(u_1,\ldots,u_p)$ multiplying by $|u_i-u_{j}|^d$ factors as needed to build $F$.

\end{proof}

\begin{thm}
\label{thm2} With the notation as above, 
$$
 \lim_{\lambda\to-1} \Gamma(\lambda)
\int\limits_{0}^{1} u_1^{(d(\lambda+1)-2)/2} (1-u_1)^{(d-2 +d(q-p))/2} F(\lambda,u_1) du_1 = cF(-1,0)$$

for a nonzero constant c.
\end{thm}

\begin{lemma} 
We perform a change of variable to make the notation nicer:
$$\lim_{\lambda\to-1} \Gamma(\lambda)
\int\limits_{0}^{1} u_1^{(d(\lambda+1)-2)/2} (1-u_1)^{(d-2 +d(q-p))/2} F(\lambda,u_1) du_1 $$
 $$=  \lim_{\lambda\to-1} \frac{2}{d} \Gamma(\lambda) 
\int\limits_{0}^{1} u_1^{\lambda} (1-u_1)^{(d-2 +d(q-p))/2} F\left( \frac{2}{d}(\lambda+1)-1,u_1\right) du_1. $$

\end{lemma}

We may ignore the $\mu(u_1) =  (1-u_1)^{(d-2 +d(q-p))/2}$ factor in the limit.  

\begin{lemma}
With notation as above, 

\begin{align*}
&\lim_{\lambda\to-1} \Gamma(\lambda) \int\limits_{0}^{1}   u^{\lambda}\mu(u) F(2/d(\lambda+1)-1,u) du\\
=&\lim_{\lambda\to-1}\Gamma(\lambda) \int\limits_{0}^{1}   u^{\lambda}  F(2/d(\lambda+1)-1,u)  du.
\end{align*}

\end{lemma} 
\begin{proof}

The idea of the proof is to break the integral over two subintervals $[0,\eta]$ and $[\eta,1]$ so that on $[0,\eta]$ we have $|1-\mu(u)|$ small.  On $[\eta,1]$, $\mu$ is integrable and the $u^\lambda$ term drives the integral to 0.

\end{proof}

\begin{proof}[Proof of Theorem \ref{thm2}]
Now we will see that 
$$\lim_{\lambda\to-1}\Gamma(\lambda)\int\limits_{0}^{1}   u^{\lambda}  F(2/d(\lambda+1)-1,u)  du = F(-1,0).$$

As before, the proof is a matter of breaking the integral into subintervals $[0,\eta]$ and $[\eta, 1]$.  We choose $\eta$ so that $|F(\lambda,u) - F(\lambda,0)|$ is uniformly small on $[0,\eta]$.  Then, as $\lambda\to -1$, the integral over $[\eta,1]$ is driven to 0 by the $u^\lambda$ factor.  

\end{proof}

\begin{cor}  In coordinates $\lim_{\lambda\to-1}\Gamma(\lambda) \C^\lambda f(\beta)$ equals
\begin{equation}
\label{09734576772304987}
c\int\limits_{0}^{1}  
  \cdots  \int\limits_{u_{p-1}}^{1}  \prod_{i=2}^p  u_i^{d-1}  f^\vee(0,u_2,\ldots,u_p) \nu_{p-1}(u_2,\ldots, u_p) du_p\cdots du_2
  \end{equation}
\end{cor}

If we reverse the change of variable $u_i = \cos^2(t_i)$ for $i=2,\ldots,p$, we get 
\begin{dmath}
\label{34373470892340897250}
c \int\limits_{0}^{\pi/2}  
  \cdots  \int\limits_{0}^{t_{p-1}} \droprank{\alpha}{1}(\textbf{t}_1)^{d}  f_1(\textbf{t}_1) \deltak{1}(\textbf{t}_1) dt_p\cdots dt_2.
  \end{dmath}

\begin{defn}  For $f \in C^\infty(\Gr(p,\K^n))$, the transform $\droprank{\funk}{1}$ is  defined by $$\droprank{\funk}{1} f =  \lim_{\lambda \to -1}\gamma(\lambda)\C^\lambda f.$$ 
\end{defn}

\subsection{Geometric Interpretation of $
\droprank{\funk}{1}$}
\label{geometry}

Our goal in this section is to arrive at a geometrically meaningful interpretation of the transform $\droprank{\funk}{1}$.  We will see that this transform can be interpreted as a cosine transform on a lower-dimensional Grassmannian $\droprank{\B}{1} \cong \Gr(p-1,\K^{n-2})$ embedded in $\B \cong \Gr(p,\K^n)$. 

\subsubsection{Cosine-$\lambda$ Transform on a Lower Rank Grassmannian}

Consider the integral (\ref{34373470892340897250}).  Recall $\textbf{t}_{1}=(t_2,\ldots,t_p)$ and  
 $Y_{1}(\textbf{t}_{1}) = Y(0,t_2,\ldots,t_p)$.  Let $k_1=\exp Y(\pi/2,0,\ldots,0)$.  With this notation, for (\ref{34373470892340897250}) we can write 

\begin{equation}
\label{8832986239998723987}
c \int\limits_{0}^{\pi/2}  
  \cdots  \int\limits_{0}^{t_{p-1}}  \droprank{\alpha}{1}(\textbf{t}_{1})^{d}   f(\exp Y_{1}(\textbf{t}_{1}) k_1 \beta)\deltak{1}(\textbf{t}_{1})  dt_p\cdots dt_2
  \end{equation}

Let $$
{K_1} = \{ x\in K | xe_1 = \gamma_1 e_1\text{ and }xe_n = \gamma_ne_n; \gamma_1\text{ and } \gamma_n \text{ scalars} \}$$ and $\Lk{1} = \Kk{1}\cap L$.  Then $\Kk{1}/\Lk{1} \cong \SU(n-2,\K) / (\SU(n-1, \K)\times \SU(n-1,\K))$.

\begin{lemma} Let $\pi : K_1 \rightarrow K/L$ denote the quotent map $x\mapsto xL$.  Then $\pi(\Kk{1})$ is a closed, embedded submanifold of $K/L$ and it is diffeomorphic to $\Kk{1}/ \Lk{1}$.
\end{lemma}

\begin{proof}
The group $\Kk{1}$ acts on $K/L$ and $\Kk{1}$ is compact, so the action is proper (see \cite{MR2954043} on proper group actions).  It follows that orbits in $K/L$ are closed embedded submanifolds.  Observe that $\pi(\Kk{1})$ is the orbit of $L$.  Thus $\pi(\Kk{1})$ is a closed manifold on which $\Kk{1}$ acts transitively, and it is clear that the stabilizer of $L$ under this action is $\Lk{1}$, so $\pi(\Kk{1})$ is diffeomorphic to $\Kk{1}/\Lk{1}$.
\end{proof}

The following proposition is very straightforward and proof is omitted.

\begin{prop}
\label{98859843876282984}
Let $\droprank{\mathfrak{b}}{1}=\{ Y_{1}(\textbf{t}_{1})| t_i\in \mathbb{R}\}$.
\begin{enumerate}
\item  The involution $\tau$ of $K$ restricts to an involution on $\Kk{1}$ and $\Kk{1}^\tau = \Lk{1}$.  We use the notation $\kfrack{1}= \droprank{\mathfrak{l}}{1} \oplus \droprank{\mathfrak{q}}{1}$ for the eigenspace decomposition of the Lie algebra.  Then $\droprank{\mathfrak{l}}{1}\subset \f l$ and $\droprank{\mathfrak{q}}{1} \subset \f q$.
\item  The subspace $\droprank{\mathfrak{b}}{1}$ is a maximal abelian subspace of $\droprank{\mathfrak{q}}{1}$.
\item The submanifold $\pi (\Kk{1})$ is a symmetric space under the action of the group $\Kk{1}$.  Its involution is the restriction of $\tau$.

\item The translate $\droprank{\B}{1} :=\pi(\Kk{1}) k_1$ is an embedded submanifold and is diffeomorphic to $\Gr(p-1,\K^n)$.
\end{enumerate}

\end{prop}

Since $\droprank{\B}{1} =\pi(\Kk{1}) k_1$ is a Grassmannian in its own right, there is a cosine-$\lambda$ transform defined on it which we denote $\Ck{1}^\lambda$.  

In line with our development above and applying Proposition \ref{98859843876282984} we can write this transform in coordinates.  Observe that in dropping down from $\Gr(p,\K^n)$ to $\Gr(p-1,\K^{n-2})$, the important value $q-p = (q-1) - (p-1)$ is preserved, so the root system falls into the same category.  We take the positive Weyl chamber on $\droprank{\mathfrak{b}}{1}$ induced by our choice on $\f b$.  We let $\droprank{\mathfrak{b}}{1}^+$ denote the positive Weyl chamber and $\droprank{\B}{1}^+= \exp \droprank{\mathfrak{b}}{1}^+$.  We take a fundamental domain $\Uk{1}$ for the map $\exp: \droprank{\mathfrak{b}}{1}^+ \rightarrow \droprank{\B}{1}^+$ in the same way as before.  We will also evaluate the $\Ck{1}^\lambda$ transform on a particular $\betak{1} = k_1 \beta$.  Then we have

\begin{align}
 \Ck{1}^\lambda f (\betak{1}) &= \int\limits_{\Uk{1}} \prod_{i=2}^{p}|\cos \epsilon_i(Y)|^{d\lambda} f(\exp Y \betak{1}) \deltak{1}(\exp Y) dY\\
&= \int\limits_{0}^{\pi/2}  \label{49876587274309825}
  \cdots  \int\limits_{0}^{t_{p-1}} \droprank{\alpha}{1}(\textbf{t}_{1})^{d\lambda} f(\exp Y_1(\textbf{t}_{1})\betak{1})    \deltak{1}(\exp Y_{1}(\textbf{t}_{1})) dt_p\ldots dt_2
\end{align}
Here $\deltak{1} = \prod\limits_{\alpha \in \rootsk{1} } |\sin \alpha(i\log(b))|^{m_\alpha}$, 
where $\rootsk{1}$ denotes the positive restricted roots of $\kfrack{1}$ 
with respect to $\droprank{\mathfrak{b}}{1}^+$.  
Proposition \ref{rootsProp} applies with the modification that the indices range between 2 and $p$. 
Therefore 
$\deltak{1} 
(\exp Y_1(\textbf{t}_1)) = 
\deltak{1} 
(\textbf{t}_1)$.

We observe now that 
$$(\ref{49876587274309825}) = \Ck{1}^1(f)(\betak{1}).$$
Note that this is a $\Ck{1}^\lambda$ transform with $\lambda = 1$.  This is the essential geometric observation.  The normalized cosine-$\lambda$ family of transforms on a Grassmannian yields at $\lambda=-1$ a transform that is itself a cosine-$\lambda$ transform on a Grassmannian of lower rank.

\subsubsection{$L$-Orbits of Lower Rank Grassmannians}

We make a geometric observation about the way these lower rank Grassmannians sit inside $\B$.  

Observe that $\droprank{\B}{1} =  \{ \sigma \in \Gr(p, \K^n)  \; | \; e_n \in \sigma \text{ and } \sigma\subset e_1^\perp \}$.  Let us set the notation 
$$Z_u^v = \{ \eta \in \Gr(p, \K^n)  \; | \; u\in\eta  \text{ and } \eta \subset v^\perp \}.$$

\begin{prop}
Given $\xi,\eta \in \Gr(p, \K^n)$, $\xi$ contains a vector orthogonal to $\eta$ if and only if  $|\Cos(\xi, \eta)|=0$ 
\end{prop}

\begin{proof}
Assume $|\Cos(\xi, \eta)|=0$.  We consider $\xi$ and $\eta$ $dp$-dimensional real vector spaces.  Given an orthonormal basis $\{ \xi_1, \ldots, \xi_{pd}\}$ for $\xi$, let $E$ denote the unit-volume parallelepiped formed with these vectors at its edges.  Let $v_i'$ denote $P_\eta v_i$, the orthogonal projection onto $\eta$.  Since $|\Cos(\xi,\eta)|=0$, we have $\Vol(P_\eta(E))=0$, so the set $\{ v_1',\ldots, v_{dp}' \}$ is linearly dependent.  Therefore the span of $\{ v_1,\ldots, v_{dp} \}$ contains some element $v$ contained in the kernel of $P_\eta$, which is $\eta^{\perp}$.  This gives us an element $v$ orthogonal to $\eta$ in the real dot product.  Since $\eta=i\eta=j\eta=k\eta$ over $\mathbb{H}$ and $\eta=i\eta$ over $\mathbb{C}$, for $d>1$ this implies that $v$ is orthogonal to $\eta$ in the hermitian form $\langle\cdot,\cdot \rangle_{\K}$ also.  

For the converse, if $\xi$ contains an element in the kernel of $P_\eta$ then the volume of $P_\eta(E)$ is clearly 0.  
\end{proof}

\begin{prop}
The action of $K$ on $\Gr(p,\K^n)$ induces an action on the family $\{ Z_\nu^\omega | \nu, \omega$ perpedicular unit vectors in $\K^n\}$ and this action is given by
$$k \cdot Z_\nu^\omega = Z_{k\nu}^{k\omega}$$
\end{prop}

For $\eta \in \Gr(p, \K^n)$, let us set $Z(\eta) = \{ \xi \in \Gr(p, \K^n) \; \vert \;\; |\Cos(\xi,\eta)| = 0 \}$.  Then $Z(\beta)$ is the set where $\eta\mapsto |\Cos(\eta,\beta)|^{-1}$ blows up.  One might think of this set $Z(\beta)$ as the appropriate notion of $\beta^{\perp}$ in the Grassmannians by analogy with $v^{\perp}$ for $v$ on the sphere.  On the sphere, $v^\perp$ is the lower dimensional sphere where $u\mapsto |\Cos(u,v)|$ takes zero values.  

Observe that
$$ Z(\beta) = \bigcup\limits_{v\in \beta,\omega \in \beta^\perp} Z_v^\omega = L Z_{e_1}^{e_{n}} = L\droprank{\B}{1}.$$
Therefore, in this sense, $Z(\beta)$ decomposes into copies of $\Gr(p-1,\K^{n-1})$, and $L$ acts on this family of lower-dimensional Grassmannians transitively.  Since we assume $f$ is $L$-invariant, nothing is lost by restricting attention to $\droprank{\B}{1}$.

\subsection{Image and Kernel of $\droprank{\funk}{1}$}
Now we turn to some representation-theoretic considerations regarding the integral transform $\funk_1$.

\begin{prop}
The transform $\funk_1$ is an intertwining operator of the left regular representation of $K$ on $\mathcal{C}^\infty(\B)$.
\end{prop}

\begin{proof}
Let $L_k$ denote left translation by $k$.
 The cosine-$\lambda$ transform $\C^\lambda$ is a mermomorphic family of intertwining operators (see \op, Theorem 4.5).  Thus for any $k\in K$ and $\lambda > -1$, we have $\C^\lambda\circ L_k - L_k\circ \C^\lambda =  0$.  It follows by the analytic continuation that the equality $\gamma(\lambda)\C^\lambda\circ L_k - L_k\circ \gamma(\lambda)\C^\lambda =  0$ extends to the limit at $\lambda=-1$.
\end{proof}

Therefore, the image and kernel of $\funk_1$ are invariant subspaces and we will characterize them in terms of $\mu$, the highest weights in the decomposition in (\ref{5772398732477098098}). 

\begin{prop}
The image of $\funk$ is composed of those subspaces $L_\mu^2(\B)$ with highest weight $\mu=(m_1,\ldots, m_p)$ where $m_2=\cdots m_p =0$.
\end{prop}

Recall the spectrum of $\C^\lambda$ and consider the function 
\begin{equation}\label{37930987698}
\frac{\eta_\mu(\lambda)}{\eta_0(\lambda)}= \left( \frac{\Gamma_{p,d}(\frac{1}{2}(-d\lambda)+\mu))}{\Gamma_{p,d}(\frac{1}{2}(-d\lambda))}\right) \left(  \frac{\Gamma_{p,d}(\frac{1}{2}(d\lambda+dn))}{\Gamma_{p,d}(\frac{1}{2}(d\lambda+dn+\mu))}\right).
\end{equation}
We examine for which values of $\mu$ this function is 0 at $\lambda=-1$ and for which values it is nonzero.

Here, recall that $\mu=(m_1,\ldots, m_p)$ and in all cases we have that the $m_i$ are all even integers and $m_1\geq \cdots \geq |m_p|$ where $m_p$ can only be negative in the case $p=q$ and $d=1$.  

Suppose $p>2$.  We expand the factor on the left in (\ref{37930987698}):
\begin{equation}\label{96827523732409}
\frac{\Gamma(\frac{d}{2}(-\lambda)+m_1/2)}{\Gamma(\frac{d}{2}(-\lambda))} 
\frac{\Gamma(\frac{d}{2}(-\lambda-1)+m_2/2)}{\Gamma(\frac{d}{2}(-\lambda-1))}
\cdots
\frac{\Gamma(\frac{d}{2}(-\lambda - p + 1)+m_p/2)}{\Gamma(\frac{d}{2}(-\lambda- p + 1))}
\end{equation}

The factor on the right in (\ref{37930987698}) expands to
\begin{equation}\label{2762430894209}\frac{\Gamma(\frac{d}{2}(\lambda+n))}{\Gamma(\frac{d}{2}(\lambda+n)+m_1/2)}
\frac{\Gamma(\frac{d}{2}(\lambda+n-1))}{\Gamma(\frac{d}{2}(\lambda+n-1)+m_2/2)}
\cdots
\frac{\Gamma(\frac{d}{2}(\lambda+q))}{\Gamma(\frac{d}{2}(\lambda+q)+m_p/2)}.
\end{equation}

Note that (\ref{2762430894209}) cannot be infinite since $q>1$, so for the product in (\ref{37930987698}) to be finite, the other factor (\ref{96827523732409}) must be nonzero.

Considering (\ref{96827523732409}), each factor $\Gamma(\frac{d}{2}(-\lambda- j))$ in the denominator for $j>0$ is infinite at $\lambda=-1$ if $\frac{d}{2}(-\lambda- j)$ is an integer.  For (\ref{96827523732409}) to be nonzero, then, each of these factors must be matched by an infinity in the numerator.  In particular, the factor $\Gamma(\frac{d}{2}(-\lambda-1))$ in the denominator must be matched by $\Gamma(\frac{d}{2}(-\lambda-1)+m_2/2)$ in the numerator, which requires that $m_2=0$.  This forces $m_2=\cdots=m_p =0$.  However, $\Gamma(\frac{d}{2}(-\lambda))$ is finite and $m_1$ is free.

Suppose $p=2$.  Then we have expansions
$$\frac{\Gamma(\frac{1}{2}(-\lambda)+m_1/2)}{\Gamma(\frac{1}{2}(-\lambda))} 
\frac{\Gamma(\frac{1}{2}(-\lambda-1)+m_2/2)}{\Gamma(\frac{1}{2}(-\lambda-1))}
$$
and
$$\frac{\Gamma(\frac{d}{2}(\lambda+n))}{\Gamma(\frac{d}{2}(\lambda+n)+m_1/2)}
\frac{\Gamma(\frac{d}{2}(\lambda+n-1))}{\Gamma(\frac{d}{2}(\lambda+n-1)+m_2/2)}$$

By the same reasoning as above, $m_2=0$ and $m_1$ is free unless it's possible for $m_2$ to be negative.  That can only happen when $p=q$ and $d=1$.  In that case, the factor $\frac{\Gamma(\frac{d}{2}(-\lambda-1)+m_2/2)}{\Gamma(\frac{d}{2}(-\lambda-1))}$ is nonzero for $m_2\leq 0$.  However, noting that $n=4$, the factor 
$$\frac{\Gamma(\frac{1}{2}(\lambda+4-1))}{\Gamma(\frac{1}{2}(\lambda+4-1)+m_2/2)}$$
is zero for $m_2 <0$ (and $m_2$ an even integer).

Thus, in all cases, the function $\frac{\eta_\mu(\lambda)}{\eta_0(\lambda)}$ is nonzero at $\lambda=-1$ precisely for those values of $\mu$ where $m_2=\cdots=m_p=0$.

\section{Higher Poles of $\C^\lambda$}

\label{higherPolessec}

\noindent
We now turn our attention to the higher poles of the cosine-$\lambda$ transform---those on the negative integers -2,\ldots, -p.  Here we restrict attention to the Grassmannians over $\reals$.  In this case, B. Rubin has eplored the analytic continuation of a normalized cosine-$\lambda$ transform to $\lambda = -p$.  Recall (\ref{rubinsac}) from the introduction.

In this section we will assume $p\geq 2$ because the $p=1$ case is well understood and because in that case there are no ``higher poles'' above $\lambda=-1$ to consider.

At first glance our work here would seem to render our previous analysis of the first pole unnecessary because it applies to that pole also.  There are two reasons for presenting both that argument and this argument separately.  Our analysis in previous sections applied whether the field was $\reals$, $\mathbb{C}$, or $\mathbb{K}$, but here we use Rubin's result, which was only proved in a setting over $\reals$.  The second reason is that our analysis in previous sections is quite different from Rubin's methods.

We translate Rubin's work into the language and notations of this paper.  Rubin works in terms of Stiefel manifolds, but as he points out we may apply his theorems to the Grassmannian picture by viewing the functions on the Stiefel manifold as right-$\O(p)$-invariant functions so they lift to the Grassmannian.  He proves that 
$$\mathop{a.c.}_{\lambda=-p} \gamma(\lambda) \C^\lambda f (\beta) = c\int_{\sigma \subset \beta^\perp} f(\sigma) d\sigma$$
in the invariant measure, where $c$ is a nonzero constant.  Note that in our view in this paper, we have fixed $\beta$ and $L$ is the stabilizer of $\beta$ in $K$.  Thus, fix any $\eta\in \beta^\perp$ and 

$$\int_{\sigma \subset \beta^\perp} f(\sigma) d\sigma = \int_{L} f(l \eta) dl = f(\eta)$$
where we assume $f$ is $L$-invariant as before.  The coordinates used above provide a convenient choice of $\eta$ given by $(t_1,\ldots,t_p) = (\pi/2,\ldots, \pi/2)$.  Then, in our view of things, this result can be stated as such: the analytic continuation of $\gamma(\lambda)\C^\lambda f(\beta)$ to $\lambda=-p$ is $f(\pi/2,\ldots,\pi/2)$ up to a non-zero factor.   

The intuitive idea of the following result is as follows.  The analytic continuation of $\gamma(\lambda)\C^\lambda f(\beta)$ to the pole at -1 yields an integral transform on an embedded sub-manifold which is itself a Grassmannian of rank $p-1$  Further, this transform has the form of a cosine-$\lambda$ transform evaluated at $\lambda= 1$.  At $\lambda=-p$, we have a simple evaluation at a point, which we may view as the rank 0 case.  At the poles in between, we step down in rank at each iteration from -1 to -p.  That is, at -2 we will have a cosine-$\lambda$ transform over an embedded submanifold which is a Grassmannian of rank p-2, evaluated at a particular $\lambda$ which comes out of the analysis, and it continues in this way.

Recall that in coordinates 

\begin{equation}
\label{4723902398249}
\C^\lambda f(\beta) = \int\limits_{0}^{1}  \int\limits_{u_1}^{1}  
  \cdots  \int\limits_{u_{p-1}}^{1}  \prod_{i=1}^p  u_i^{(d(\lambda+1)-2)/2}  f^\vee(\textbf{u})  \nu_p(\textbf{u}) du_p\cdots du_1
  \end{equation}

\begin{lemma}
Fix $k=1, \ldots, p-1$ and let 
\begin{multline}
F(\lambda, u_1,\ldots, u_{p-k}) = \int\limits_{u_{p-k}}^{1}   
  \cdots  \int\limits_{u_{p-1}}^{1}  \prod_{i=1}^p  u_i^{(d(\lambda+1)-2)/2}  f^\vee(\textbf{u}) \\ \prod_{i\leq p-k<j} |u_i-u_j|  \droprank{\nu}{1}(\textbf{u}) du_{p}\cdots du_{p-k+1}
 \end{multline}
which is the inner $k$ integrals in (\ref{4723902398249}).  Then $F$ is uniformly continuous at $\lambda = -p+k$.

\end{lemma}

\begin{proof}
The innermost integral is 
$$F_1(\lambda,u_{p-1}) = \int_{u_{p-1}}^1u_p^{\frac{\lambda-1}{2}} \prod_{i<p} |u_i-u_{p}|  f^\vee(\textbf{u}) \droprank{\nu}{1}(\textbf{u}) du_p.$$
We have suppressed the dependence on the other variables. 
Observe that $u_p^{\frac{-p+k-1}{2}} \prod_{i<p} |u_i-u_{p}|$ is bounded for $u_p\in [0,1]$ since $u_1\leq \cdots \leq u_p$.  This is because 
$$u_p^{\frac{-p+k-1}{2}} \prod_{i<p} |u_i-u_{p}|\leq  u_p^{\frac{-p+k-1}{2}} u_p^{p-1} $$
which is bounded when $0\leq p-2$, and we have assumed $p\geq 2$.

Then $|F_1(-p+k, u_{p-1})- F_1(-p+k+h, u_{p-1}) | \leq  \int_{u_{p-1}}^1 u_p^{-p+k}(1-u_p^h)  \prod_{i<j} |u_i-u_{p}|\; |f(u_{p-1},u_p)| du_p \leq M  \int_{u_{p-1}}^1(1-u_p^h)du_p$, some $M$. 

  We iterate outward like this observing that at each stage  $u_j^{\frac{-p+k-1}{2}} \prod_{i<j} |u_i-u_{j}|$ is bounded.

\end{proof}

\begin{thm}

With $k$, $F$ and $\gamma$ as above,

$$ \ac_{\lambda = -p+k} \gamma(\lambda) C^\lambda f (\beta) = F(-p+k,0,\ldots, 0)$$
which equals
\begin{equation}
\label{870982988248} 
\int\limits_{0}^{\pi/2}  
  \cdots  \int\limits_{0}^{t_{p-1}}  \prod_{i=p-k}^p  |\cos t_i|^{p-k}  \fk{k}(\textbf{t}_k) \deltak{k}(\textbf{t}_k) dt_p\cdots dt_{p-k}.
\end{equation}
\end{thm}

\begin{proof}
We view $\C^\lambda f(\beta)$ as a cosine-$\lambda$ transform  $\Ck{k}^\lambda$ of a function $F$ defined on a Grassmannian manifold of rank $p-k$. 
In this case, it is evaluated at $\beta_{k}$, the element spanned by $\{b_1,\ldots, b_{p-k}\}$.
Rubin proved that
$$\ac_{\lambda = -p+k} \gamma(\lambda)\Ck{k}^\lambda[F(-p+k, u_1, \ldots, u_{p-k})](\beta_{k}) = c F(-p+k, 0, \ldots,0)$$ 
where $c$ is a nonzero constant.  

For $\lambda$ close enough to  $-p+k$, we have $$|F(\lambda, u_1, \ldots, u_{p-k}) - F(-p+k, u_1, \ldots, u_{p-k})| < \epsilon$$ uniformly.  By linearity, and supressing dependence except on $\lambda$,

$$\Ck{k}^\lambda[F(\lambda)] = \Ck{k}^\lambda[F(-p+k)] + \Ck{k}^\lambda[F(\lambda) - F(-p+k)    ].$$
Then as $\lambda \to -p+k$, it follows that
$$\gamma(\lambda)\Ck{k}^\lambda[F(\lambda) - F(-p+k)    ](\beta_{k})\to 0.$$
The integral (\ref{870982988248}) is just an evaluation of  $F(-p+k,0,\ldots, 0)$ followed by a change of variable.  
\end{proof}

\end{document}